\documentclass[12pt]{amsart}
\usepackage{amsmath}
\usepackage[bmargin=1.5in]{geometry}
\usepackage{amsthm}
\usepackage{amsfonts}
\usepackage{comment}
\usepackage{amssymb}
\usepackage{eufrak}
\usepackage{comment}
\usepackage{hyperref}
\usepackage{mathtools}

\numberwithin{equation}{section}
\usepackage{breqn}

\newtheorem{Theorem}{Theorem}[section]

\newtheorem{Lemma}[Theorem]{Lemma}

\theoremstyle{definition}
\newtheorem*{Example}{Example}
\theoremstyle{remark}
\newtheorem*{remark}{Remark}

    \usepackage{tikz}
    \usetikzlibrary{arrows,matrix}

\newcommand{\ord}{\text{ord}}
\newcommand{\s}{s}
\newcommand{\zz}{\mathbb{Z}}
\newcommand{\qq}{\mathbb{Q}}
\newcommand{\B}{\text{\bf B}_2}

\newcommand{\Title}[1]{\begin{center}\Large{Title???}
\normalsize  \end{center}}
\usepackage{cite}

\title{Modular units from quotients of Rogers-Ramanujan type $q$-series}

\author[Hannah Larson]{Hannah Larson}
\address{5015 Donald St., Eugene, OR, 97405}
\email{hannahlarson@college.harvard.edu}

\begin{document}

\maketitle

\begin{abstract}
In \cite{F1} and \cite{F2}, Folsom presents a family of modular units as higher-level analogues of the Rogers-Ramanujan $q$-continued fraction. These units are constructed from analytic solutions to the higher-order $q$-recurrence equations of Selberg. Here, we consider another family of modular units,  which are quotients of Hall-Littlewood $q$-series that appear in the generalized Rogers-Ramanujan type identities of \cite{GOW}. In analogy with the results of Folsom, we provide a formula for the rank of the subgroup these units generate and show that their specializations at the cusp $0$ generate a subgroup of the cyclotomic unit group of the same rank. In addition, we prove that their singular values generate the same class fields as those of Folsom's units.
\end{abstract}

\section{Introduction}

The Rogers-Ramanujan $q$-continued fraction
\begin{equation}
r(\tau) := \dfrac{q^{1/5}}{1 + \dfrac{q}{1 + \dfrac{q^2}{1 + \dfrac{q^3}{\ddots}}}},
\end{equation}
where $q := e^{2\pi i \tau}$
has the surprising property that its values at imaginary quadratic irrational points $\tau$ in the upper-half of the complex plane, known as \textit{CM points}, are algebraic integer units \cite{BCZ}. For example, Ramanujan famously evaluated \cite{BR}
\[r(i) = \sqrt{\frac{5+\sqrt{5}}{2}} - \frac{\sqrt{5}+1}{2}.\]
This function arises from the quotient of the two $q$-series in the famous Rogers-Ramanujan identities,
\begin{equation} \label{RR1}
\sum_{n=0}^\infty\frac{q^{n^2}}{(1-q) \cdots (1 - q^n)} = \prod_{n=0}^\infty \frac{1}{(1 - q^{5n+1})(1 - q^{5n+4})}
\end{equation}
and
\begin{equation} \label{RR2}
\sum_{n=0}^\infty \frac{q^{n^2+n}}{(1 - q) \cdots (1 - q^n)} = \prod_{n=0}^\infty \frac{1}{(1 - q^{5n+2})(1 - q^{5n+3})}.
\end{equation}
One exhibits the continued fraction expansion from the left-hand sides by observing that the function
\begin{equation} \label{sel}
S(z) = S(z;q) := \sum_{n= 0}^\infty \frac{z^n q^{n^2}}{(1 - q) \cdots (1 - q^n)}
\end{equation}
satisfies the $q$-recurrence $S(z) = S(zq) + zqS(zq^2)$. Meanwhile, the right-hand sides of \eqref{RR1} and \eqref{RR2} give rise to an infinite product expansion
\[r(\tau) =q^{1/5}\prod_{n=0}^\infty \frac{(1 - q^{5n+1})(1-q^{5n+4})}{(1 - q^{5n+2})(1 - q^{5n+3})} ,\]
 which shows $r(\tau)$ to be a modular unit of level $5$.

Here, a \textit{modular unit} is a modular function with no zeros or poles in the upper half of the complex plane. For a given level $\ell$, such functions constitute the unit group $U_\ell$ of the integral closure of the ring $\qq[j] \subset \mathcal{F}_\ell$ where
\[j(\tau) = q^{-1} + 744 + 196884q + \ldots\]
is the classical modular invariant and $\mathcal{F}_\ell$ is the field of modular functions for $\Gamma(\ell)$ with Fourier expansion defined over $\qq(\zeta_\ell)$. In this paper, we will consider various subgroups of the unit group $U_\ell$, all of which will be denoted by $U_\ell^*$ with different superscripts $*$.

Let $\ell = 2k+1 \geq 5$ be an odd integer. In \cite{F1} and \cite{F2}, Folsom introduces a family of modular units $r_{\ell, m}(\tau)$ for $1 \leq m \leq k-1$ as higher-level analogues of the Rogers-Ramanujan $q$-continued fraction. The functions $r_{\ell, m}(\tau)$ are essentially quotients of Selberg functions $S_k(q)$, which are functional solutions to higher-order $q$-recurrences, generalizing \eqref{sel}.

Folsom proves that the $r_{\ell, m}(\tau)$ satisfy several properties similar to the Rogers-Ramanujan $q$-continued fraction, regarding the subgroup $U_\ell^C$ of the unit group they generate, their specializations at the cusp $0$, and the extensions generated by their values at CM points, known as \textit{singular values}. The aim of this paper is to prove analogous results for another family of modular units $s_{\ell,m}(\tau)$ that also arise naturally as generalizations of the Rogers-Ramanujan $q$-continued fraction.

In \cite{GOW}, the authors present a framework for Rogers-Ramanujan type identities and introduce infinite families of Rogers-Ramanujan type $q$-series generalizing \eqref{RR1} and \eqref{RR2}. 
These expressions involve the standard symbols
\[(a)_k = (a;q)_k := \begin{dcases} (1-a)(1-aq) \cdots (1-aq^{k-1}) & \text{if $k \geq 0$} \\
\prod_{j=0}^\infty (1-aq^j) & \text{if $k=\infty$}\end{dcases} \]
and
\[ \theta(a;q) := (a;q)_\infty(q/a;q)_\infty \qquad \text{with} \qquad \theta(a_1,\ldots, a_n; q) := \theta(a_1;q)\cdots \theta(a_n;q).\]
For positive integers $m$ and $n$, if $\ell = 2m + 2n + 1$, the authors show (see Theorem 1.1 in \cite{GOW}) that
\begin{equation} \label{gRR1}
\sum_{\lambda:\lambda_1 \leq m} q^{|\lambda|}P_{2\lambda}(1,q,q^2,\ldots;q^{2n-1}) = \frac{(q^\ell;q^\ell)^n_\infty}{(q)_\infty^n} \prod_{i=1}^n \theta(q^{i+m};q^\ell) \prod_{1\leq i < j \leq n} \theta(q^{j-1},q^{i+j-1};q^\ell)
\end{equation}
and
\begin{equation} \label{gRR2}
\sum_{\lambda:\lambda_1 \leq m} q^{2|\lambda|}P_{2\lambda}(1,q,q^2,\ldots;q^{2n-1}) =
\frac{(q^\ell;q^\ell)^n_\infty}{(q)_\infty^n} \prod_{i=1}^n \theta(q^{i};q^\ell) \prod_{1\leq i < j \leq n} \theta(q^{j-1},q^{i+j};q^\ell),
\end{equation}
where the sums range over partitions $\lambda$ and the summands include their associated Hall-Littlewood polynomials $P_{2\lambda}(x_1, x_2, \ldots;q)$. The authors then define appropriate normalizations of these series,
\begin{equation} \label{generalRR1}
\Phi_{1a}(m,n;\tau) := q^{\frac{mn(4mn-4m+2n-3)}{12\ell}} \sum_{\lambda:\lambda_1\leq m} q^{|\lambda|}P_{2\lambda}(1,q,q^2,\ldots;q^{2n-1})
\end{equation}   
and
\begin{equation} \label{generalRR2}
\Phi_{1b}(m,n;\tau) := q^{\frac{mn(4mn+2m+2n+3)}{12\ell}} \sum_{\lambda:\lambda_1 \leq m} q^{2|\lambda|}
P_{2\lambda}(1,q,q^2,\ldots;q^{2n-1}).
\end{equation}
Using the infinite product sides of \eqref{gRR1} and \eqref{gRR2}, the authors arrive at the following expression for the quotient of these two series in terms of Siegel functions $g_a(\tau)$ (defined in \eqref{Sprod}),
\begin{equation} \label{defPsi}
\Psi_1(m, n;\tau) := \frac{\Phi_{1a}(m, n;\tau)}{\Phi_{1b}(m, n;\tau)} = \prod_{j=1}^m \frac{g_{(\frac{2j}{\ell}, 0)}(\ell \tau)}{g_{(\frac{j}{\ell}, 0)}(\ell \tau)}.
\end{equation}

We first establish that these functions are indeed modular units.
\begin{Theorem} \label{pu}
For positive integers $m$ and $n$, the functions $\Psi_1(m, n;\tau)$ are modular units of level $\ell = 2m+2n+1$.
\end{Theorem}

\begin{Example}
When $m=n=2$, applying \eqref{gRR1} shows that
\[
\Phi_{1a}(2,2;\tau) = q^{1/3}\sum_{\lambda:\lambda_1\leq 2} q^{|\lambda|}P_{2\lambda}(1,q,q^2,\ldots;q^3) = q^{1/3}\prod_{n=1}^\infty \frac{(1-q^{9n})}{(1-q^n)}, \]
which is the $q$-series in Dyson's favorite identity, recalled in his ``A walk through Ramanujan's garden" \cite{D}.
Meanwhile, using \eqref{gRR2}, the other $q$-series is
\[\Phi_{1b}(2,2;\tau) = q \sum_{\lambda:\lambda_1\leq 2} q^{2|\lambda|}P_{2\lambda}(1,q,q^2,\ldots;q^3) = q\prod_{n=1}^\infty \frac{(1-q^{9n})(1-q^{9n-1})(1-q^{9n-8})}{(1-q^n)(1-q^{9n-4})(1-q^{9n-5})}.
\]
Taking the quotient of these two series, we find that
\[\Psi_1(2,2;\tau) = q^{-2/3}\prod_{n=1}^\infty \frac{(1-q^{9n-1})(1-q^{9n-8})}{(1-q^{9n-4})(1-q^{9n-5})},\]
which is a modular function of level $9$.
\end{Example}

In analogy with Folsom's notation, we will now write
\[s_{\ell,m}(\tau) := \Psi_1\left(m, \frac{\ell-1}{2} - m; \tau\right).\]
In Theorem 1.7 (3) of \cite{GOW}, the authors show that for any CM point $\tau$, the value $s_{\ell,m}(\tau)$ is an algebraic integer unit. Although Folsom does not show the analogous result for the $r_{\ell,m}(\tau)$, it is not hard to prove from the theory she constructed.

\begin{Theorem} \label{au}
Let $\ell = 2k+1 \geq 5$ be an odd integer and let $2 \leq m \leq k$ with $\text{gcd}(\ell, m) = 1$. For any CM point $\tau$, we have that $r_{\ell, k+1-m}(\tau)$ is an algebraic integer unit. When $\ell$ is prime this implies that the singular values of any function in the subgroup $U_\ell^+ \subseteq U_\ell$ defined in Section \ref{ranksec} are algebraic integer units.
\end{Theorem}

In Theorem 1.iii of \cite{F1}, Folsom studies the subgroup generated by the $r_{\ell,m}(\tau)$,
\[U_\ell^C := \langle r_{\ell,m}(\tau) : 1 \leq m \leq k-1\rangle,\]
and proves that $\mathrm{Rank}(U^C_{\ell}) = k-1$. It then follows that, together with the constant functions, $U_\ell^C$ generates the subgroup $U_\ell^+$ of modular units whose zeros and poles are supported on certain cusps (see Section \ref{ranksec}). Similarly, we define $U_\ell^R$ to be the subgroup of the unit group generated by the $s_{\ell,m}(\tau)$,
\[U_\ell^R := \langle s_{\ell,m}(\tau) : 1 \leq m \leq k-1 \rangle.\]
It is natural to ask if the same is true of $U_\ell^R$. We will show that this is the case under certain conditions.

\begin{Theorem} \label{rank}
Let $\ell = 2k+1 \geq 5$ be prime and let $d$ be the order of $2$ in $(\zz/\ell \zz)^\times/\{\pm 1\}$. Then $U_\ell^R \subseteq U_\ell^+$ and has rank
\[\mathrm{Rank}(U_\ell^R)= k - \frac{k}{d}.\]
In particular, whenever $2$ generates $(\zz/\ell \zz)^\times/\{\pm 1\}$, we have $U_\ell^R \times \qq(\zeta_\ell)^\times=U_\ell^+$.
\end{Theorem}

\begin{remark}
In Theorem 1.iii of \cite{F1}, Folsom writes $U_\ell^C = U_\ell^+$, which leaves implicit the inclusion of the constant functions $\qq(\zeta_\ell)^\times \subset U_\ell$.
\end{remark}

\begin{Example}
When $\ell = 5$, we find that $2$ has order $2$ in $(\zz/5\zz)^\times/\{\pm 1\}$ and thus
\[\text{Rank}(U_5^R) = \text{Rank}(U_5^C) = 1.\]
In this case, both groups are generated by the Rogers-Ramanujan $q$-continued fraction $r(\tau)$ which satisfies
\[\frac{1}{r(\tau)} = s_{5,1}(\tau) = -r_{5, 1}(\tau).\]
When $\ell=17$, the element $2$ has order $4$ in $(\zz/17\zz)^\times/\{\pm 1\}$ so $\mathrm{Rank}(U_{17}^R) = 6$ while $\text{Rank}(U_{17}^C) = 7$.
\end{Example}

In Theorem 1.2 of \cite{F2}, Folsom shows that for $\ell$ prime, the limiting values of the functions $r_{\ell,m}(\tau)$ at the cusp $0$ are generators for the cyclotomic unit group of $\qq(\zeta_\ell)^+$. Because the $r_{\ell,m}(\tau)$ generate $U_\ell^+$, we will be able to express the $s_{\ell, m}(\tau)$ in terms of the $r_{\ell,m}(\tau)$ and make an analogous statement.

\begin{Theorem} \label{cy}
Specializing the functions $\s_{\ell,m}(\tau)$ to the cusp $0$ produces the values
\[\lim_{\tau\rightarrow 0}\s_{\ell,m}(\tau) = \prod_{j=1}^m \zeta_\ell^{-\frac{j}{2}}\frac{1-\zeta_\ell^{2j}}{1-\zeta_\ell^j}\]
lying in the cyclotomic unit group of $\qq(\zeta_\ell)^+$. The rank of the subgroup generated by these cyclotomic units is equal to $\mathrm{Rank}(U_\ell^R)$.
\end{Theorem}

\begin{Example}
When $\ell=5$, we have
\[\lim_{\tau \rightarrow 0} s_{5,1}(\tau) = \zeta_5^{-\frac{1}{2}}\frac{1-\zeta_5^2}{1-\zeta_5} = \phi,\]
where $\phi$ is the golden ratio. In addition, $\phi$ generates the cyclotomic unit group of the number field $\qq(\zeta_5)^+$.
\end{Example}

Finally, Folsom shows that for CM points $\tau$, the singular values $r_{\ell, m}(\tau)$ generate particular ray class fields of $\qq(\tau)$ from its Hilbert class field. We show that the singular values $\s_{\ell,m}(\tau)$ generate the same extensions.

\begin{Theorem} \label{cm}
Let $\tau$ be a CM point and let $K = \qq(\tau)$. Let $\ell = 2k+1\geq 5$ be prime. Then for any integer $1 \leq m \leq k-1$, the ray class field $K_\ell$ of modulus $\ell$ is generated over the Hilbert class field by $\s_{\ell,m}(\tau)$. That is,
\[K_\ell = K(j(\tau), \s_{\ell,m}(\tau)).\]
\end{Theorem}

We give an example of this theorem in the case $\ell = 5$.
\begin{Example}
Let $\rho$ be the primitive cube root of unity in the upper half-plane. From an evaluation due to Ramanujan, we have that
\[s_{5,1}(\rho) = \frac{1}{r(\rho)} = \frac{4\zeta_{10}}{\sqrt{30+6\sqrt{5}}-3-\sqrt{5}}.\]
Meanwhile, we have $j(\rho) = 0 \in \qq(\rho)$. The ray class field of modulus $5$ is therefore generated by $s_{5,1}(\rho)$, giving
\[\qq(\rho)_5 = \qq(\rho, s_{5,1}(\rho)).\]
\end{Example}

\begin{remark}
A recent paper of Rains and Warnaar \cite{RW} gives further identities of a type similar to those in \eqref{gRR1} and \eqref{gRR2}. Although their algebraic properties are not studied, we expect similar theorems to hold for functions arising from appropriate quotients of these $q$-series.
\end{remark}

This paper is organized as follows. In the next section, we recall the theory of Siegel functions developed by Kubert and Lang in \cite{KL}. We then prove a lemma using their $q$-series that will allow us to prove Theorems \ref{pu} and \ref{au}. In Section 3, we use the results of \cite{F2} regarding independence of Siegel functions to prove the rank formula in Theorem \ref{rank}. Drawing on the cyclotomic theory in \cite{F2}, we proceed to prove the corresponding statement in Theorem \ref{cy}. In Section 4, we recall the Galois action for the field extension $\mathcal{F}_\ell/\mathcal{F}_1$, determine the subgroup fixing the functions $\s_{\ell,m}(\tau)$, and prove Theorem \ref{cm}.

\section*{Acknowledgements}
This research was carried out during the 2015 REU at Emory University. The author would like to thank Ken Ono for suggesting this problem and providing guidance, Michael Mertens, Sarah Trebat-Leader, and Michael Griffin for useful conversations, and the NSF for its support.

\section{Siegel functions}

The proofs of our theorems will rely on the theory of Siegel functions as developed by Kubert and Lang. Both families of units turn out to have nice expressions in terms of these functions, which are in a sense the building blocks of modular units. 

For any positive integer $\ell$, the \textit{principal congruence subgroup} $\Gamma(\ell) \subseteq \text{SL}_2(\zz)$ is defined by
\[\Gamma(\ell) = \left\{\left(\begin{array}{cc} a & b \\ c & d \end{array}\right) \in \text{SL}_2(\zz) : 
\left(\begin{array}{cc} a & b \\ c & d \end{array}\right) \equiv \left(\begin{array}{cc} 1 & 0 \\ 0 & 1 \end{array}\right) \pmod \ell
 \right\}.\]
 These groups act on the upper-half plane $\mathbb{H}$ by $\gamma\tau := \frac{a\tau + b}{c\tau + d}$ where $\gamma = {\tiny \left(\begin{array}{cc} a & b \\ c & d \end{array}\right)}$. A meromorphic function $f$ defined on $\mathbb{H}$ is a 
 \textit{modular function of level $\ell$} if it is invariant under this action, i.e. if $f(\gamma\tau) = f(\tau)$ for all $\gamma \in \Gamma(\ell)$, and it is meromorphic at the cusps. The set of such functions form a field. We denote by $\mathcal{F}_\ell$ the subfield of these functions whose Fourier expansions are defined over $\qq(\zeta_\ell)$ where $\zeta_\ell := e^{2\pi i/\ell}$.

\subsection{Basic facts about Siegel functions}
Let $\B(z):= \{z\}^2 - \{z\} + \frac{1}{6}$ be the second Bernoulli polynomial evaluated at the fractional part of its argument and set $e(x) := e^{2\pi i x}$ . For each $a = (a_1,a_2) \in \frac{1}{\ell} \mathbb{Z} \backslash \mathbb{Z}$, we define the \textit{Siegel function} $g_a(\tau)$ by the $q$-series
\begin{equation} \label{Sprod}
g_a(\tau) := -q^{\frac{1}{2}\B(a_1)}e(a_2(a_1-1)/2) \prod_{n=1}^\infty (1 - q^{n-1+a_1}e(a_2))(1-q^{n-a_1}e(-a_2)),
\end{equation}
The \textit{Klein function} $\mathfrak{k}_a$ is defined as
\[\mathfrak{k}_a(\tau) := \frac{g_a(\tau)}{\eta(\tau)^2},\]
where $\eta(\tau) := q^{1/24} \prod_{n=1}^\infty(1 - q^n)$ is the \textit{Dedekind $\eta$-function}. The Klein functions satisfy the transformation equation
\begin{equation} \label{k-trans}
\mathfrak{k}_a(\gamma\tau) = (c\tau+d) ^{-1}\mathfrak{k}_{a\gamma}(\tau).
\end{equation}
Although $g_a(\tau)$ is not modular for $\Gamma(\ell)$, the above transformation property, together with the transformation properties of the Dedekind $\eta$-function, give rise to the following criterion for determining when a particular product of Siegel functions is. The Siegel functions have no zeros or poles in $\mathbb{H}$ so such functions will be units.

\begin{Theorem}[Ch.~3, Thm.~5.2 of \cite{KL}] \label{puhammer}
Let $\ell \geq 5$ be an odd integer and let $\{m(a)\}_{a \in \frac{1}{\ell} \zz^2 \backslash \zz^2}$ be a set of integers. Then the product of Siegel functions
\[\prod_{a \in \frac{1}{\ell} \zz^2 \backslash \zz^2} g_a(\tau)^{m(a)}\]
is a modular unit of level $\ell$ if and only if
\[\sum_{a\in\frac{1}{\ell}\zz^2 \backslash \zz^2} m(a)a_1^2 \equiv \sum_{a\in\frac{1}{\ell}\zz^2 \backslash \zz^2} m(a)a_2^2 \equiv \sum_{a\in\frac{1}{\ell}\zz^2 \backslash \zz^2} m(a)a_1a_2 \equiv 0 \pmod \ell\]
and
\[\sum_{a\in\frac{1}{\ell}\zz^2 \backslash \zz^2} m(a) \equiv 0 \pmod {12}.\] 
\end{Theorem}
The next result concerns the algebraic properties of the singular values of certain quotients of Siegel functions.
\begin{Theorem}[Ch.~1, Thm.~2.2 of \cite{KL}] \label{auhammer}
Let $\ell$ be an integer and suppose $a \in \frac{1}{\ell}\zz^2$ has exact period $\ell$ mod $\zz^2$. If $c \in \zz$ with $\text{gcd}(c, \ell) = 1$ then $g_{ca}(\tau)/g_a(\tau)$ is a unit over $\zz[j(\tau)]$ for any CM point $\tau$.
\end{Theorem}

\begin{remark}
By unit over $\zz[j(\tau)]$ we mean a unit in the ring $\zz[j(\tau),g_{ca}(\tau)/g_a(\tau)]$. In particular, since $j(\tau)$ is an algebraic integer, this implies $g_{ca}(\tau)/g_a(\tau)$ is an algebraic integer unit.
\end{remark}

\subsection{Proof of Theorems \ref{pu} and \ref{au}}
The expressions for the modular units we are interested in will involve the following types of products of Siegel functions.
For fixed $\ell = 2k+1$, let
\[g(m) := \prod_{s=0}^{\ell-1} g_{(m/\ell, s/\ell)}(\tau).\]
In this notation, Theorem 1.ii of \cite{F1} becomes
\begin{equation} \label{rs}
r_{\ell, k+1-m}(\tau) = (-1)^{m-1}e\left(-\frac{k(m-1)}{2\ell}\right) \frac{g(m)}{g(1)}.
\end{equation}
The following lemma will allow us to easily move between equivalent descriptions of our units.
\begin{Lemma}
For any integer $m$, we have
\[g(m) = e\left(\frac{k(m-\ell)}{2\ell}\right)g_{(m/\ell, 0)}(\ell \tau).\]
In particular, we can write
\begin{equation} \label{r1}
r_{\ell, k+1-m}(\tau) = (-1)^{m-1} \frac{g_{(m/\ell,0)}(\ell \tau)}{g_{(1/\ell,0)}(\ell\tau)}
\end{equation}
and
\begin{equation} \label{seq}
s_{\ell,m}(\tau) = e\left(-\frac{km(m+1)}{4\ell}\right)\prod_{j=1}^m \frac{g(2j)}{g(j)}.
\end{equation}
\end{Lemma}
\begin{proof}
Using the product expansion for the Siegel functions in \eqref{Sprod}, we find that
\begin{align*}
g(m) &= \prod_{s=0}^{\ell-1}-q^{\frac{1}{2}\B(\frac{m}{\ell})}e\left(\frac{s}{2\ell}\left(\frac{m}{\ell}-1\right)\right) \prod_{n=1}^\infty\left(1-q^{n-1+\frac{m}{\ell}}e\left(\frac{s}{\ell}\right)\right)
\left(1-q^{n-\frac{m}{\ell}}e\left(-\frac{s}{\ell}\right)\right).
\end{align*}
Since
\[x^{\ell} - 1 = \prod_{s=0}^{\ell-1}\left(x-e\left(\frac{s}{\ell}\right)\right) ,\]
we see that the elementary symmetric functions in the $\ell$th roots of unity all vanish except the last, so for any $y$ we have
\[\prod_{s=0}^{\ell-1} \left(x - y e \left(\frac{s}{\ell}\right)\right) = x^{\ell} - y^{\ell}.\]
Applying this above, we see
\begin{align*}
g(m) &= -q^{\frac{\ell}{2}\B(\frac{m}{\ell})} e\left(\frac{(\ell-1)(m-\ell)}{4\ell} \right) \prod_{n=1}^\infty \left(1-q^{\ell(n-1+\frac{m}{\ell})}\right)\left(1 - q^{\ell(n-\frac{m}{\ell})}\right) \\
&= e\left(\frac{k(m-\ell)}{2\ell}\right)g_{(m/\ell,0)}(\ell \tau).
\end{align*}
Putting this together with equations \eqref{rs} and \eqref{defPsi} results in the desired expressions for $r_{\ell,k+1-m}(\tau)$ and $s_{\ell,m}(\tau)$ respectively.
\end{proof}

We can now prove Theorems \ref{au} and \ref{pu}.

\begin{proof}[Proof of Theorem \ref{pu}]
By \eqref{seq}, we can express $\s_{\ell,m}(\tau)$ in terms of pure Siegel functions as
\[s_{\ell,m}(\tau) = e\left(-\frac{km(m+1)}{4\ell}\right)\prod_{j=1}^m \prod_{s=0}^{\ell-1} \frac{g_{(2j/\ell, s/\ell)}(\tau)}{g_{(j/\ell,s/\ell)}(\tau)}.\]
We then compute
\[\sum_{j=1}^{m} \sum_{s=0}^{\ell-1} (2j - j) = \frac{\ell m(m+1)}{2} \equiv 0 \pmod \ell\]
\[\sum_{j=1}^m \sum_{s=0}^{\ell-1} (s - s) = 0\]
\[\sum_{j=1}^m\sum_{s=0}^{\ell-1} (2js - js) = \frac{m(m+1)}{2} \frac{\ell(\ell-1)}{2} \equiv 0 \pmod \ell\]
and the sum of the multiplicities is zero. Thus by Theorem \ref{puhammer}, the $\s_{\ell,m}(\tau)$ are modular units of level $\ell$.
\end{proof}

\begin{proof}[Proof of Theorem \ref{au}]
When $\text{gcd}(m, \ell) = 1$, Theorem \ref{auhammer} together with \eqref{r1} tells us that $r_{\ell, k+1-m}(\tau)$ is an algebraic integer unit. If $\ell$ is prime, then the functions $r_{\ell,m}(\tau)$ generate $U_\ell^+$ (Theorem 1.iii of \cite{F1}) and all of the $r_{\ell, m}(\tau)$ will be algebraic integer units so any product of them is as well.
\end{proof}

\section{The rank of $U_\ell^R$} \label{ranksec}

To prove Theorem \ref{rank} we will study the divisors of our functions on the modular curve $X(\ell)$. Theorem \ref{cy} will then follow using the limiting values computed by Folsom in \cite{F2}.

In this section, we fix $\ell = 2k+1 \geq 5$ prime. We will need the following congruence subgroups of $\text{SL}_2(\zz)$, defined by
\begin{align*}
\Gamma_1(\ell) &:= \left\{\left(\begin{array}{cc} a & b \\ c & d \end{array}\right) \in \text{SL}_2(\zz) : 
\left(\begin{array}{cc} a & b \\ c & d \end{array}\right) \equiv \left(\begin{array}{cc} 1 & 0 \\ * & 1 \end{array}\right) \pmod \ell
 \right\} \\
\Gamma_0(\ell) &:=\left\{\left(\begin{array}{cc} a & b \\ c & d \end{array}\right) \in \text{SL}_2(\zz) : 
\left(\begin{array}{cc} a & b \\ c & d \end{array}\right) \equiv \left(\begin{array}{cc} * & * \\ 0 & * \end{array}\right) \pmod \ell
 \right\}.
\end{align*}
The \textit{modular curves} $X(\ell), X_1(\ell),$ and $X_0(\ell)$ are defined to be the quotient of the extended upper-half plane $\mathbb{H}^* = \mathbb{H} \cup \mathbb{Q} \cup \{\infty\}$ by the action of the corresponding congruence subgroup. They are all compact Riemann surfaces.

Let $\pi: X(\ell) \rightarrow X_0(\ell)$ be the canonical projection and let $\mathcal{A}_\ell = \pi^{-1}(\infty)$. Folsom proves that the functions $g(m)$ all have the same order of vanishing at cusps not in $\mathcal{A}_\ell$. More precisely, if $\text{ord}_\beta f$ is the smallest power of $q^{1/\ell}$ appearing in the Fourier expansion of $f$ at the cusp $\beta$, we have the following.
\begin{Lemma}[Folsom, Proposition 5 of \cite{F1}] \label{Fprop5}
Let $\ell \geq 5$ be prime. Then for cusps $\beta, \beta'$ not in $\mathcal{A}_\ell$ and any integers $m, n$ we have
\[\text{ord}_\beta~g(m) = \text{ord}_{\beta'}~g(n).\]
\end{Lemma}

\begin{remark}
Although Folsom states this proposition for $1 \leq m, n \leq k$, the proof is valid for any integers $m$ and $n$.
\end{remark}

Let $U_\ell^+$ be the subgroup of $U_\ell$ consisting of functions whose zeros and poles are supported on $\mathcal{A}_\ell$,
\[U_\ell^+ := \{f \in U_\ell : \text{supp}(f) \subseteq \mathcal{A}_\ell\}.\]
Lemma \ref{Fprop5} implies that any quotient of Siegel functions $\frac{g(m)}{g(n)}$ is in $U_\ell^{+}$.
The quotient of $U_\ell^+$ by the subgroup $\qq(\zeta_\ell)^\times$ of constant functions is a free abelian group of rank $|\mathcal{A}_\ell| - 1 = k-1$. Associating to each function $f$ the row vector corresponding to its divisor
\[V(f) := (\ldots, \text{ord}_\beta f, \ldots) \in \zz^k,\]
where $\beta$ runs over cusps of $\mathcal{A}_\ell$,
produces an isomorphism of $U_\ell^+/\qq(\zeta_\ell)^\times$ with the subgroup of vectors in the free abelian group on $\mathcal{A}_\ell$ whose entries sum to $0$.

To determine the rank of the groups $U_\ell^R$, we will need to find the order of vanishing of $g(m)$ at each of the cusps of $\mathcal{A}_\ell$. We can view the map $\pi:X(\ell) \rightarrow X_0(\ell)$ as the composition of maps $\pi_1: X(\ell) \rightarrow X_1(\ell)$ and $\pi_2: X_1(\ell) \rightarrow X_0(\ell)$. We may take as coset representatives for $\Gamma(\ell)$ in $\Gamma_1(\ell)$ the matrices $\left(\begin{array}{cc} 1 & b \\ 0 & 1 \end{array}\right)$ for $b = 0, \ldots, \ell-1$. Since these matrices all stabilize $\infty$, the map $\pi_1$ is one-to-one over $\infty$. Thus we may identify $\mathcal{A}_\ell$ with the cusps in $\pi_2^{-1}(\infty)$. Here, we may take as coset representatives for $\Gamma_1(\ell)$ in $\Gamma_0(\ell)$ the matrices
\[\gamma(i) := \left(\begin{array}{cc} i & b \\ \ell & i^{-1} \end{array}\right)\]
defined for $i \in (\zz/\ell\zz)^\times$. Accounting for the fact that $-1$ acts trivially on the upper-half plane, we see that the matrices $\gamma(i)$ for $i = 1, \ldots, k$ are in direct correspondence with cusps of $\mathcal{A}_\ell$ through associating $\gamma(i)$ with the cusp $\gamma(i)\infty$. We will write $\ord_if := \ord_{\gamma(i)\infty}f$ for the order of a function $f$ at the cusp $\gamma(i)\infty$. Let $V(m) = V(g(m))$ be the row vector $(\ord_1g(m), \ldots, \ord_kg(m))$. The following lemma gives an explicit expression for $V(m)$ and shows that it only depends on the class of $m$ in $(\zz/\ell\zz)^\times/\{\pm 1\}$.

\begin{Lemma} \label{div}
We have that
\[V(m) = \tfrac{\ell}{2}\left(\B\left(\tfrac{m}{\ell}\right), \B\left(\tfrac{2m}{\ell}\right), \ldots, \B\left(\tfrac{km}{\ell}\right)\right),\]
where as always the Bernoulli polynomial is evaluated at the fractional part of its argument. Therefore, $V(m) = V(-m)$.
\end{Lemma}

\begin{proof}
Looking at \eqref{Sprod}, we see that $\ord_\infty g_{(m/\ell,*)}(\tau) = \frac{1}{2}\B(\frac{m}{\ell})$.
For any $i = 1, \ldots, k$, we now show
\begin{align*}
\ord_i g(m) &= \sum_{s=0}^{\ell-1}\ord_i  g_{(m/\ell, s/\ell)}(\tau) = \sum_{s=0}^{\ell-1}\ord_\infty g_{(m/\ell,s/\ell)}(\gamma(i)\tau) \\
&= \sum_{s=0}^{\ell-1} \ord_\infty g_{(m/\ell, s/\ell)\gamma(i)}(\tau)=\sum_{s=0}^{\ell-1}g_{(im/\ell+s,*)}(\tau) =\tfrac{\ell}{2}\B(\tfrac{im}{\ell}).
\end{align*}
The second statement follows from the fact that $\B(x) = \B(1-x)$.
\end{proof}

In Proposition 6 of \cite{F1}, Folsom uses the Frobenius determinant relation and the non-vanishing of generalized Bernoulli numbers to show that the vectors $V(m)$ for $m = 1, \ldots, k$ are linearly independent. This allows us to prove the following formula.

\begin{Lemma} \label{rkUc}
Let $\ell \geq 5$ be prime. For any $c \not\equiv 0 \pmod \ell$, the rank of the $k \times k$ matrix with rows $V(cm) - V(m)$ for $m=1,\ldots, k$ is equal to $k - \frac{k}{d}$ where $d$ is the order of $c$ in the group $(\zz/\ell\zz)^\times/\{\pm 1\}$.
\end{Lemma}

\begin{proof}
By Lemma \ref{div}, the collection $\{V(cm) : 1 \leq m \leq k\}$ is a permutation of $\{V(m) : 1 \leq m \leq k\}$. Let $A$ be the matrix formed by the rows $V(m)$ for $1 \leq m \leq k$ and let $C$ be the permutation matrix for multiplication by $c$, i.e. $C=\rho(c)$ where $\rho:(\zz/\ell \zz)^\times \rightarrow \text{GL}_k(\mathbb{C})$ is the regular representation. Then the matrix we are considering is $(C-I)A$.
 Since $\text{Rank}(A)=k$, we have $\text{Rank}((C-I)A) =\text{Rank}(C - I)$. The kernel of $C-I$ is the eigenspace of $C$ for eigenvalue $1$ which has dimension equal to the number of cycles in $C$ which is $\frac{k}{d}$. Therefore, $\text{Rank}((C-I)A) = k-\frac{k}{d}$.
\end{proof}

To complete the proof that $U_\ell^C \times \qq(\zeta_\ell)^\times = U_\ell^+$, Folsom shows that $U_\ell^C \times \qq(\zeta_\ell)$ and $U_\ell^+$ are both cotorsion free in $U_\ell$ (Proposition 3 of \cite{F1}). In general, an identical argument shows the following.

\begin{Lemma} \label{co}
Suppose we have a collection of modular units
\[h_i(\tau) = \prod_{m=1}^{\ell-1} g(m)^{\epsilon_i(m)}\]
with $\epsilon_i(m) = 0, \pm 1$ and $\sum_m \epsilon_i(m) = 0$. Then the subgroup of $U_\ell$ generated by $\qq(\zeta_\ell)^\times$ and the $h_i(\tau)$ is cotorsion free. In particular, the group $U_\ell^R \times \qq(\zeta_\ell)^\times$ is cotorsion free.
\end{Lemma}

Theorem \ref{rank} now follows from Lemmas \ref{rkUc} and \ref{co}.

\begin{proof}[Proof of Theorem \ref{rank}]
By \eqref{seq} we have
\[V(s_{\ell,m}(\tau)) = \sum_{j=1}^m V(2j) - V(j).\]
The rank of the subgroup generated by the functions $s_{\ell,m}(\tau)$ is equal to the rank of the $k \times k$ matrix with rows $V(s_{\ell,m}(\tau))$, which is equal to the rank of the $k\times k$ matrix with rows $V(2m) - V(m)$. Applying Lemma \ref{rkUc} shows
\[\mathrm{Rank}(U_\ell^R) = k - \frac{k}{d},\]
where $d$ is the order of $2$ in $(\zz/\ell\zz)^\times/\{\pm 1\}$. In the case when $2$ is a generator, we have $d=k$ so 
\[\mathrm{Rank}(U_\ell^R) = k-1 = \text{Rank}(U_\ell^+/\qq(\zeta_\ell)^\times).\]
Thus, by Lemma \ref{co} we have
\[U_\ell^R \times \qq(\zeta_\ell)^\times= U_\ell^+. \qedhere\]
\end{proof}

We now discuss the parallel result for the specializations of these functions in the cusp $0$. Let $E_{\qq(\zeta_\ell)^+}^C$ be the cyclotomic unit group of $\qq(\zeta_\ell)^+$, the maximal real subfield of $\qq(\zeta_\ell)$. It is a standard result (see Lemma 8.1 of \cite{W}) that
\[E_{\qq(\zeta_\ell)^+}^C = \left\langle -1, \zeta_\ell^{\frac{1-m}{2}} \frac{1-\zeta_\ell^m}{1-\zeta_\ell} : 1 < m < \tfrac{\ell}{2} \right\rangle. \]
In \cite{F2} (4.15), Folsom shows that
\[\lim_{\tau \rightarrow 0} r_{\ell,m}(\tau) = (-1)^{k-m} \zeta_\ell^{\frac{1-v}{2}}\frac{\zeta_\ell^v-1}{\zeta_\ell-1},\]
where $v = k+1 - m$, and thus together with $-1$, these limiting values generate $E_{\qq(\zeta_\ell)^+}^C$. In fact, since the $r_{\ell,m}(\tau)$ are independent, the map 
\[\lim: U_\ell^C\times \{\pm 1\}  \rightarrow E_{\qq(\zeta_\ell)^+}^C\]
determined by sending each function to its limiting value at $0$ is an isomorphism of groups. Theorem \ref{cy} now follows.

\begin{proof}[Proof of Theorem \ref{cy}]
By \eqref{defPsi} and \eqref{r1}, we have
\begin{align*}
s_{\ell,m}(\tau) = \prod_{j=1}^{m} (-1)^{j}\frac{r_{\ell,k+1-2j}(\tau)}{r_{\ell,k+1-j}(\tau)}.
\end{align*}
and so
\[\lim_{\tau\rightarrow 0} s_{\ell,m}(\tau) = 
\prod_{j=1}^m(-1)^j \frac{(-1)^{2j-1}\zeta_\ell^{\frac{1-2j}{2}}\frac{\zeta_\ell^{2j}-1}{\zeta_\ell-1}}{(-1)^{j-1}\zeta_\ell^{\frac{1-j}{2}}\frac{\zeta_\ell^j-1}{\zeta_\ell-1}}
=\prod_{j=1}^m \zeta_{\ell}^{-\frac{j}{2}}\frac{\zeta_\ell^{2j} - 1}{\zeta_\ell^j - 1}.
\]
Since the map $\lim:U_\ell^C\times \{\pm 1\} \rightarrow E_{\qq(\zeta_\ell)^+}^C$ described above is an isomorphism, the image of $U_\ell^R$ has the same rank.
\end{proof}

\section{Singular values of $s_{\ell,m}(\tau)$}

In this section, we prove Theorem \ref{cm}. We begin by recalling the Galois action of the field extension $\mathcal{F}_\ell/\mathcal{F}_1$ and in particular the action on Siegel functions. We then determine the subgroup of $\text{Gal}(\mathcal{F}_\ell/\mathcal{F}_1)$ fixing the function $s_{\ell,m}(\tau)$. This will show that the functions $s_{\ell, m}(\tau)$ and $r_{\ell,m}(\tau)$ generate the same extensions of the modular function field, which turns out to be those functions of level $\ell$ with Fourier expansion defined over $\qq$. After this point, the methods of Folsom's proof that the singular values of $r_{\ell,m}(\tau)$ generate ray class fields apply to show that the same is true for the singular values of $s_{\ell,m}(\tau)$.

\subsection{Galois action on Siegel functions}
The field $\mathcal{F}_\ell$ is a Galois extension of $\mathcal{F}_1$ with Galois group isomorphic to $\mathrm{GL}_2(\zz/\ell\zz)/\{\pm I\}$.
The natural action of $\mathrm{SL}_2(\zz/\ell\zz)$ on $\mathcal{F}_\ell$ is given by
\[\gamma \cdot f(\tau) := f(\gamma\tau). \]
Meanwhile, for $\gamma_d := \left(\begin{array}{cc} 1 & 0 \\ 0 & d\end{array}\right)$ with $d \in (\zz/\ell\zz)^\times$ we let
\[\gamma_d \cdot \zeta_\ell = \zeta_\ell^d\]
and extend this to $\mathcal{F}_\ell$ by acting on Fourier coefficients. That is, if $f \in \mathcal{F}_\ell$ has Fourier expansion $f(\tau) = \sum_{n=m}^\infty a_nq^{n/\ell}$, then we define
\[\gamma_d\cdot f(\tau) := \sum_{n=m}^\infty(\gamma_d\cdot a_n)q^{n/\ell}. \]
Since the matrices $\gamma_d$ together with $\mathrm{SL}_2(\zz/\ell\zz)$ generate $\text{GL}_2(\zz/\ell\zz)$, this determines the Galois action. From \eqref{Sprod} and \eqref{k-trans}, one may confirm that the Galois group acts on Siegel functions by multiplication on indices. That is, for $\gamma \in \mathrm{GL}_2(\zz/\ell\zz)$ we have
\begin{equation} \label{act}
\gamma\cdot g_{a}(\tau) = g_{a\gamma}(\tau).
\end{equation}

\subsection{Proof of Theorem \ref{cm}}
Folsom's proof that the singular values of $r_{\ell,m}(\tau)$ generate the ray class fields over the Hilbert class fields depends only on the fact that these functions are generators of a specific extension of $\mathcal{F}_1$, namely the fixed field of the elements $\gamma_d$. To establish Theorem \ref{cy}, it thus suffices to prove the following.

\begin{Lemma} \label{cmkey}
For $1 \leq m \leq k-1$, we have that $\mathcal{F}_1(s_{\ell,m}(\tau))$ is the fixed field of the subgroup $\{\gamma_d: d \in (\zz/\ell\zz)^\times\}$.
\end{Lemma}
\begin{proof}
We argue in a fashion similar to Folsom's proof of Proposition 7.2 in \cite{F2}.
Suppose that $\gamma = \left(\begin{array}{cc} a & b \\ c & d \end{array}\right) \in \mathrm{GL}_2(\zz/\ell\zz) \simeq \text{Gal}(\mathcal{F}_\ell/\mathcal{F}_1)$ fixes $s_{\ell,m}(\tau)$. We first show that $c \equiv 0 \pmod \ell$. 
Note that from \eqref{defPsi} and \eqref{Sprod} we have
\begin{align*}
\ord_\infty s_{\ell,m}(\tau) &= \frac{\ell^2}{2}\sum_{j=1}^m \B(\tfrac{2j}{\ell}) - \B(\tfrac{j}{\ell}) = \frac{1}{2}\sum_{j=1}^m 3j^2-\ell j \\
&= \frac{3}{2}\cdot \frac{m(m+1)(2m+1)}{6} - \frac{\ell}{2}\cdot \frac{m(m+1)}{2} \\
&= \frac{m(m+1)}{4}(2m+1-\ell) \neq 0.
\end{align*}
By \eqref{seq} and \eqref{act}, we have
\begin{align*}
\gamma \cdot s_{\ell,m}(\tau) = \alpha \prod_{j=1}^m \prod_{s=0}^{\ell-1} \frac{g_{(2j/\ell, s/\ell)\gamma}(\tau)}{g_{(j/\ell,s/\ell)\gamma}(\tau)} = \alpha  \prod_{j=1}^m\prod_{s=0}^{\ell-1} \frac{g_{(\frac{2aj+cs}{\ell}, \frac{2bj+ds}{\ell})}(\tau)}{g_{(\frac{aj+cs}{\ell},\frac{bj+ds}{\ell})}(\tau)},
\end{align*}
where $\alpha$ is some root of unity. Therefore,
\begin{equation*}\label{ord}
\ord_\infty (\gamma\cdot s_{\ell,m}(\tau)) = \frac{\ell}{2}\sum_{j=1}^m\sum_{s=0}^{\ell-1} \B\left(\tfrac{2aj+cs}{\ell}\right) - \B\left(\tfrac{aj+cs}{\ell}\right).
\end{equation*}
If $c \not\equiv 0 \pmod \ell$ then as $s$ ranges from $0$ to $\ell-1$, the numbers  $2aj+cs$ and $aj+cs$ both run over all residues mod $\ell$, and we would find $\ord_\infty(\gamma\cdot s_{\ell,m}(\tau)) = 0$. 
Hence, we must have $c \equiv 0 \pmod \ell$. Then the above becomes
\begin{align*}
\ord_\infty (\gamma\cdot s_{\ell,m}(\tau)) &= \frac{\ell}{2}\sum_{j=1}^m\sum_{s=0}^{\ell-1} \B\left(\tfrac{2aj}{\ell}\right) - \B\left(\tfrac{aj}{\ell}\right) = \frac{1}{2}\sum_{j=1}^m 3a^2j^2-\ell a j \\
&= \frac{3a^2}{2}\cdot \frac{m(m+1)(2m+1)}{6} - \frac{\ell a}{2}\frac{m(m+1)}{2} \\
&=\frac{m(m+1)}{4}(a^2(2m+1)-\ell a).
\end{align*}
For $\gamma$ to fix $s_{\ell,m}(\tau)$ we must have $\ord_{\infty}(\gamma \cdot s_{\ell,m}(\tau)) = \ord_\infty s_{\ell,m}(\tau)$. This implies 
\[a^2(2m+1) - a\ell = 2m+1 -\ell,\]
 and reducing both sides mod $\ell$ we find $a^2 \equiv 1 \pmod \ell$. Hence, $a \equiv \pm 1 \pmod \ell$ and since $\gamma$ and $-\gamma$ are identified in $\text{Gal}(\mathcal{F}_\ell/\mathcal{F}_1)$ we may take $a \equiv 1 \pmod \ell$.
 
 We have now determined that $\gamma$ is of the form
 \[\gamma = \left(\begin{array}{cc} 1 & b \\ 0 & d \end{array}\right) = \left(\begin{array}{cc} 1 & 0 \\ 0 & d \end{array}\right)\left(\begin{array}{cc} 1 & b \\ 0 & 1 \end{array}\right).\]
 Note that from \eqref{Sprod} we have the following $q$-series expansion
 \begin{equation} \label{over-q}
 \frac{g_{(\frac{2j}{\ell},0)}(\ell\tau)}{g_{(\frac{j}{\ell},0)}(\ell\tau)} = q^{\frac{3j^2-\ell j}{2\ell}}\prod_{n=1}^\infty \frac{(1-q^{\ell n - \ell + 2j})(1 - q^{\ell n - j})}{(1-q^{\ell n - \ell + j})(1 - q^{\ell n - j})}.
 \end{equation}
In particular, the Fourier expansion of $s_{\ell,m}(\tau)$ at infinity is defined over $\qq$ and hence fixed by $\gamma_d$. From the $q$-series above we see
\[\left(\begin{array}{cc} 1 & b \\ 0 & 1 \end{array}\right)\cdot \frac{g_{(\frac{2j}{\ell},0)}(\ell\tau)}{g_{(\frac{j}{\ell},0)}(\ell\tau)}  = \frac{g_{(\frac{2j}{\ell},0)}(\ell(\tau+b))}{g_{(\frac{j}{\ell},0)}(\ell(\tau+b))} = e\left(\frac{b(3j^2-\ell j}{2\ell}\right) \frac{g_{(\frac{2j}{\ell},0)}(\ell\tau)}{g_{(\frac{j}{\ell},0)}(\ell\tau)},\]
and so
\[\gamma\cdot s_{\ell,m}(\tau) = \prod_{j=1}^m  e\left(\frac{b(3j^2-\ell j}{2\ell}\right) \frac{g_{(\frac{2j}{\ell},0)}(\ell\tau)}{g_{(\frac{j}{\ell},0)}(\ell\tau)} = e\left(\sum_{j=1}^m \frac{b(3j^2-\ell j)}{2\ell}\right) s_{\ell, m}(\tau).\]
For $\gamma$ to fix $s_{\ell,m}(\tau)$ we must have
\[\sum_{j=1}^m \frac{b(3j^2-\ell j)}{2\ell} = \frac{b}{\ell} \cdot \frac{m(m+1)}{4}(2m+1-\ell) \]
be an integer. Since $1 \leq m \leq k-1$, we must have $b \equiv 0 \pmod \ell$. Thus we have shown that if $\gamma \in \mathrm{GL}_2(\zz/\ell\zz)$ fixes $s_{\ell, m}(\tau)$ then $\gamma =\gamma_d$ for some $d \in (\zz/ \ell\zz)^\times$. Conversely, since $s_{\ell,m}(\tau)$ has Fourier expansion defined over $\qq$, given in \eqref{over-q}, it is fixed by all matrices of this form.
\end{proof}

\begin{proof}[Proof of Theorem \ref{cm}]
Proposition 7.2 of \cite{F2} shows that $\mathcal{F}_1(r_{\ell,m})$ is also the fixed field of $\{\gamma_d : d \in (\zz/\ell\zz)^\times\}$. Thus, applying Proposition 7.3 of \cite{F2} in the same way to $s_{\ell,m}(\tau)$ yields the result that $K_\ell = K(j(\tau), s_{\ell,m}(\tau))$.
\end{proof}

\end{document}